\documentclass[oneside,english]{amsart}
\usepackage[T1]{fontenc}
\usepackage[latin9]{inputenc}
\usepackage{verbatim}
\usepackage{float}
\usepackage{multirow}
\usepackage{amstext}
\usepackage{amsthm}
\usepackage{amssymb}
\usepackage{graphicx}

\makeatletter

\providecommand{\tabularnewline}{\\}

\numberwithin{equation}{section}
\numberwithin{figure}{section}
\theoremstyle{plain}
\newtheorem{thm}{\protect\theoremname}
  \theoremstyle{definition}
  \newtheorem{defn}[thm]{\protect\definitionname}
  \theoremstyle{plain}
  \newtheorem{lem}[thm]{\protect\lemmaname}
  \theoremstyle{plain}
  \newtheorem{prop}[thm]{\protect\propositionname}
  \theoremstyle{definition}
  \newtheorem{example}[thm]{\protect\examplename}
  \theoremstyle{remark}
  \newtheorem*{rem*}{\protect\remarkname}
  \theoremstyle{plain}
  \newtheorem{conjecture}[thm]{\protect\conjecturename}

\usepackage{array}

\author{Lin Jiu}
\address{Department of Mathematics and Statistics, Dalhousie University, 
6316 Coburg Road, Halifax, Nova Scotia, Canada B3H 4R2}
\email{Lin.Jiu@dal.ca}

\author{Diane Yahui Shi*}
\thanks{*Corresponding author}
\address{School of Mathematics, Tianjin University, No. 92 Weijin Road, 
Nankai District, Tianjin 300072, P. R. China}
\email{shiyahui@tju.edu.cn}

\date{}

\DeclareMathOperator{\sech}{sech}

\makeatother

\usepackage{babel}
  \providecommand{\conjecturename}{Conjecture}
  \providecommand{\definitionname}{Definition}
  \providecommand{\examplename}{Example}
  \providecommand{\lemmaname}{Lemma}
  \providecommand{\propositionname}{Proposition}
  \providecommand{\remarkname}{Remark}
\providecommand{\theoremname}{Theorem}

\begin{document}

\title[Orthogonal Polynomials and Lattice Path]{Orthogonal Polynomials and Lattice Path Interpretation for Higher-order
Euler Polynomials}
\begin{abstract}
We study the higher-order Euler polynomials and give the corresponding
monic orthogonal polynomials, which are Meixner-Pollaczek polynomials
with certain arguments and constant factors. Moreover, through a general
connection between moments of random variables and the generalized
Motzkin numbers, we can obtain a new recurrence formula and a matrix
representation for the higher-order Euler polynomials, interpreting
them as weighted lattice paths. 
\end{abstract}


\maketitle

\section{\label{sec:Introduction}Introduction}

As a generalization of \emph{Euler numbers} $E_{n}$ and \emph{Euler
polynomials} $E_{n}(x)$, the \emph{higher-order Euler polynomials}
$E_{n}^{(p)}(x)$ are defined by
\begin{equation}
\left(\frac{2}{e^{z}+1}\right)^{p}e^{xz}=\sum_{n=0}^{\infty}E_{n}^{(p)}(x)\frac{z^{n}}{n!},\label{eq:GFEpn}
\end{equation}
with special values $E_{n}^{(1)}(x)=E_{n}(x)$ and $E_{n}^{(1)}(1/2)=E_{n}(1/2)=E_{n}/2^{n}$.
(See, e.g., \cite[Entries 24.16.3 and 24.2.9]{NIST}.) The first several
terms of $E_{n}^{(p)}(x)$ are as follows.

\vspace{-8bp}

\begin{table}[H]
\begin{centering}
\begin{tabular}{|l||l|l|l|}
\hline 
\multirow{1}{*}{} & $p=1$  & $p=2$  & $p=3$ \tabularnewline
\hline 
$n=0$  & $1$  & $1$  & $1$ \tabularnewline
$n=1$  & $x-\frac{1}{2}$ & $x-1$  & $x-\frac{3}{2}$\tabularnewline
$n=2$  & $x^{2}-x$ & $x^{2}-2x+\frac{1}{2}$  & $x^{2}-3x+\frac{3}{2}$\tabularnewline
$n=3$  & $x^{3}-\frac{3}{2}x^{2}+\frac{1}{4}$ & $x^{3}-3x^{2}+\frac{3}{2}x+\frac{1}{2}$  & $x^{3}-\frac{9}{2}x^{2}+\frac{9}{2}x$\tabularnewline
$n=4$  & $x^{4}-2x^{3}+x$ & $x^{4}-4x^{3}+3x^{2}+2x-1$  & $x^{4}-6x^{3}+9x^{2}-3$\tabularnewline
\hline 
\end{tabular}\\
 ~ 
\par\end{centering}
\caption{\label{tab:EpnFirstValues}$E_{n}^{(p)}(x)$ for $0\leq n\leq4$ and
$1\leq p\leq3$}
\end{table}
\vspace{-25bp}

Our first result here is to give the monic orthogonal polynomials
with respect to $E_{n}^{(p)}(x)$. Given a sequence $m_{n}$, one
can define the monic orthogonal polynomials $P_{n}(y)$ with respect
to $m_{n}$ as follows. For integers $r$ and $n$ with $0\leq r<n$,
the orthogonality means
\begin{equation}
y^{r}P_{n}(y)|{}_{y^{k}=m_{k}}=0,\label{eq:Orthogonality}
\end{equation}
where the left-hand side means expanding the polynomial and evaluating
as $y^{k}=m_{k}$ on each power of $y$. In addition, $P_{n}$ satisfies
a three-term recurrence \cite[p.~47]{Shohat}: for certain sequences
$s_{n}$ and $t_{n}$, $P_{0}(y)=1$, $P_{1}(y)=y-s_{0}$, and when
$n\geq1$, 

\begin{equation}
P_{n+1}(y)=(y-s_{n})P_{n}(y)-t_{n}P_{n-1}(y).\label{eq:ThreeTermRec}
\end{equation}

After Touchard \cite[eq.~44]{Touchard} computed the monic orthogonal
polynomials with respect to the \emph{Bernoulli numbers} $B_{n}$,
Carlitz \cite[eq.~4.7]{Calitz} and also with Al-Salam \cite[p.~93]{AlsalamCarlitz}
gave the monic orthogonal polynomials, denoted by $Q_{n}(y)$, with
respect to $E_{n}$ as follows. $Q_{0}(y)=1$, $Q_{1}(y)=y$ and for
$n\geq1$, 
\begin{equation}
Q_{n+1}(y)=yQ_{n}(y)+n^{2}Q_{n-1}(y).\label{eq:RecQ}
\end{equation}
Now, let $\Omega_{n}^{(p)}(y)$ be the monic orthogonal polynomials
with respect to $E_{n}^{(p)}(x)$, i.e., similarly as (\ref{eq:Orthogonality}),
for integers $r$ and $n$, with $0\leq r<n$, 
\begin{equation}
y^{r}\Omega_{n}^{(p)}(y)|_{y^{k}=E_{k}^{(p)}(x)}=0.\label{eq:OrthogonalityOmege}
\end{equation}
Our first result is to give the recurrence of $\Omega_{n}^{(p)}(y)$.
\begin{thm}
\label{thm:Orthogonal}For integer $p\geq1$, we have $\Omega_{0}^{(p)}(y)=1$,
$\Omega_{1}^{(p)}(y)=y-x+p/2$ and 
\begin{equation}
\Omega_{n+1}^{(p)}(y)=\left(y-x+\frac{p}{2}\right)\Omega_{n}^{(p)}(y)+\frac{n(n+p-1)}{4}\Omega_{n-1}^{(p)}(y).\label{eq:RecOmega}
\end{equation}
\end{thm}
See Example \ref{exa:Orthogonality} to illustrate the orthogonality
(\ref{eq:OrthogonalityOmege}) for $p=n=2$. Furthermore, Theorem
\ref{thm:Orthogonal} links $E_{n}^{(p)}(x)$ to the generalized Motzkin
numbers, defined next.
\begin{defn}
\label{def:GMN} Given arbitrary sequences $\sigma_{k}$ and $\tau_{k}$,
the \emph{generalized Motzkin numbers} $M_{n,k}$ are defined by $M_{0,0}=1$
and for $n>0$ by the recurrence 
\begin{equation}
M_{n+1,k}=M_{n,k-1}+\sigma_{k}M_{n,k}+\tau_{k+1}M_{n,k+1},\label{eq:MotzkinNumberRec}
\end{equation}
where $M_{n,k}=0$ if $k>n$ or $k<0$. (See also \cite[eq.~3]{Motzkin}.)
\end{defn}
The second result identifies $E_{n}^{(p)}(x)$ as the generalized
Motzkin numbers, which allows us to endow new recurrence, as (\ref{eq:MotzkinNumberRec})
for $E_{n}^{(p)}(x)$. 
\begin{thm}
\label{thm:Motzkin}Let $\mathfrak{E}_{n,k}^{(p)}$ be the generalized
Motzkin numbers with special choices $\sigma_{k}=x-p/2$ and $\tau_{k}=-k\left(k+p-1\right)/4$.
Then, $\mathfrak{E}_{n,0}^{(p)}=E_{n}^{(p)}(x).$
\end{thm}
To prove both Theorem \ref{thm:Orthogonal} and Theorem \ref{thm:Motzkin},
we shall organize this paper as follows.

In Section \ref{sec:Orthogonality}, we first review some basic definitions
and properties on random variables, orthogonal polynomials, and also
the probabilistic interpretation for $E_{n}^{(p)}(x)$, viewing them
as moments of a certain random variable. Next, instead of proving
the recurrence (\ref{eq:RecOmega}), we recognize $\Omega_{n}^{(p)}(y)$
as the Meixner-Pollaczek polynomials, whose definition and important
properties will also be introduced in this section. 

In Section \ref{sec:GMN}, after introducing a combinatorial interpretation
of the generalized Motzkin numbers, we present two continued fractions
expressions. This leads to a general theorem, Theorem \ref{thm:generalTHM},
identifying moments of a random variable and the generalized Motzkin
numbers. Then, we see that Theorem \ref{thm:Motzkin} is just a special
case on $E_{n}^{(p)}(x)$. Moreover, the combinatorial interpretation
as weighted lattice paths provides a matrix representation for $E_{n}^{(p)}(x)$.
In the end, an example presents a connection between the Euler numbers
and Catalan numbers. 

In the last section, Section \ref{sec:Bernoulli}, we give analogues
for Bernoulli polynomials and a conjecture for the higher-order Bernoulli
polynomials.

\section{\label{sec:Orthogonality}Orthogonal polynomials for higher-order
Euler polynomials}

\subsection{Preliminaries}

We first recall some necessary definitions and classical results on
random variables.

Given an arbitrary random variable $X$ on $\mathbb{R}$, with probability
density function $p(t)$ and moments $m_{n}$, namely, $m_{n}=\mathbb{E}[X^{n}]=\int_{\mathbb{R}}t^{n}p(t)\mathrm{d}t$,
one can consider the \emph{monic orthogonal polynomials} with respect
to $X$, denoted by $P_{n}(y)$, which are monic polynomials with
degree $\deg P_{n}=n$. For positive integers $u$ and $v$, 
\[
\mathbb{E}\left[P_{u}(X)P_{v}(X)\right]=\int_{\mathbb{R}}P_{u}(t)P_{v}(t)p(t)\mathrm{d}t=c_{u}\delta_{u,v}
\]
(see \cite[eq.~2.20]{Shohat}), where $c_{u}$ are constants depending
on $u$ and $\delta_{u,v}$ is the Kronecker delta function, which
gives $1$ if $u=v$, and $0$ if $u\neq v$. Equivalently, the orthogonality
can be expressed as a system of equations: for integers $r$ and $n$
with $0\leq r<n$, 
\[
\mathbb{E}\left[X^{r}P_{n}(X)\right]=\int_{\mathbb{R}}t^{r}P_{n}(t)\mathrm{d}t=y^{r}P_{n}(y)|_{y^{k}=m_{k}}=0,
\]
which is the same as (\ref{eq:Orthogonality}). The three-term recurrence
of $P_{n}(y)$ is stated in (\ref{eq:ThreeTermRec}), with $P_{0}(y)=1$
and $P_{1}(y)=y-s_{0}$. 

If $X'$ is another random variable independent of $X$ with moments
$m'_{n}$, then for the random variable $X+X'$ we have 
\begin{equation}
\mathbb{E}\left[\left(X+X'\right)^{n}\right]=\sum_{k=0}^{n}\binom{n}{k}m_{k}m'_{n-k}=\left(y_{1}+y_{2}\right)^{n}|_{y_{1}^{k}=m_{k},y_{2}^{k}=m'_{k}}.\label{eq:Convolution}
\end{equation}

The next lemma gives the moments and the monic orthogonal polynomials
after shifting or scaling $X$, which is crucial in the proof of Theorem
\ref{thm:Orthogonal}.
\begin{lem}
\label{lem:KEY} Let $C$ and $c$ be constants. 

$1$. For the shifted random variable $X+c$, the corresponding moments
are
\[
\mathbb{E}\left[\left(X+c\right)^{n}\right]=\sum_{k=0}^{n}\binom{n}{k}m_{k}c^{n-k}
\]
 and the monic orthogonal polynomials, denoted by $\bar{P}_{n}(y)$,
satisfy $\bar{P}_{0}(y)=1$, $\bar{P}_{1}(y)=y-s_{0}-c$ and for $n\geq1$,
\begin{equation}
\bar{P}_{n+1}(y)=(y-s_{n}-c)\bar{P}_{n}(y)-t_{n}\bar{P}_{n-1}(y).\label{eq:RecPBar}
\end{equation}

$2$. For the scaled random variable $CX$, the moments are $\mathbb{E}\left[\left(CX\right)^{n}\right]=C^{n}m_{n}$,
and the monic orthogonal polynomials, denoted by $\tilde{P}_{n}(y)$,
satisfy $\tilde{P}_{0}(y)=1$, $\tilde{P}_{1}(y)=y-Cs_{0}$ and for
$n\geq1$, 
\begin{equation}
\tilde{P}_{n+1}(y)=(y-Cs_{n})\tilde{P}_{n}(y)-C^{2}t_{n}\tilde{P}_{n-1}(y).\label{eq:RecPt}
\end{equation}
\end{lem}
\begin{proof}
Computations for the moments are straightforward. We only consider
the monic orthogonal polynomials.

1. For $X+c$, notice that $\bar{P}_{n}(y):=P_{n}(y-c)$ satisfies
\[
\mathbb{E}\left[\bar{P}_{u}(X+c)\bar{P}_{v}(X+c)\right]=\mathbb{E}\left[\bar{P}_{u}(X)\bar{P}_{v}(X)\right]=c_{u}\delta_{u,v}.
\]
The recurrence (\ref{eq:RecPBar}) follows by shifting $y\mapsto y-c$
in (\ref{eq:ThreeTermRec}).

2. Similarly, for $CX$, $\tilde{P}_{n}(y):=C^{n}P_{n}(y/C)$. 
\end{proof}
Next, we recall the probabilistic interpretation for $E_{n}^{(p)}(x)$.
See, e.g., \cite[eq.~2.6]{Euler}. 

Let $L_{E}$ be a random variable with density function $p_{E}(t):=\sech(\pi t)$
on $\mathbb{R}$. Also consider a sequence of independent and identically
distributed (i.\ i.\ d.\ ) random variables $\left(L_{E_{i}}\right)_{i=1}^{p}$
with each $L_{E_{i}}$ having the same distribution as $L_{E}$. Then
$E_{n}^{(p)}(x)$ is the $n$th moment of a certain random variable:
\[
E_{n}^{(p)}(x)=\mathbb{E}\left[\left(x+\sum_{i=1}^{p}iL_{E_{i}}-\frac{p}{2}\right)^{n}\right].
\]
For simplicity, we denote $\epsilon_{i}=iL_{E_{i}}$ and $\epsilon^{(p)}:=\sum_{i=1}^{p}\epsilon_{i}$.
Then $(\epsilon_{i})_{i=1}^{n}$ is also an i.~i.~d.~sequence.
Moreover, 

\begin{equation}
E_{n}^{(p)}(x)=\mathbb{E}\left[\left(x+\epsilon^{(p)}-\frac{p}{2}\right)^{n}\right].\label{eq:MomentEpn}
\end{equation}
The \emph{higher-order Euler numbers} are usually defined as $E_{n}^{(p)}:=E_{n}^{(p)}(0)$,
for $p>1$. We next define another sequence of numbers, related to
$E_{n}^{(p)}(x)$ and $E_{n}^{(p)}$, as follows.
\begin{defn}
Define the sequence $\bar{E}_{n}^{(p)}$ by the exponential generating
function 
\[
\sech^{p}(z)=\sum_{n=0}^{\infty}\bar{E}_{n}^{(p)}\frac{z^{n}}{n!}.
\]
From (\ref{eq:GFEpn}) and (\ref{eq:MomentEpn}), we see that 
\begin{equation}
\bar{E}_{n}^{(p)}=2^{n}E_{n}^{(p)}\left(\frac{p}{2}\right),\ \ \ \bar{E}_{n}^{(1)}=E_{n},\ \ \ \text{and}\ \ \ \bar{E}_{n}^{(p)}=\mathbb{E}\left[\left(2\epsilon^{(p)}\right)^{n}\right].\label{eq:EulerPN}
\end{equation}
\end{defn}
As stated in Section \ref{sec:Introduction}, the orthogonal polynomials
with respect to $E_{n}$, denoted by $Q_{n}(y)$ satisfy the recurrence
(\ref{eq:RecQ}). In fact, the next definition shows that $Q_{n}(y)$
are basically the Meixner-Pollaczek polynomials; see, e.g., \cite[eq.~9.7.1]{OrthPoly}.
\begin{defn}
The \emph{Meixner-Pollaczek polynomials} are defined by 
\[
P_{n}^{(\lambda)}(y;\phi):=\frac{(2\lambda)_{n}}{n!}e^{in\phi}\,_{2}F_{1}\left(\genfrac{}{}{0pt}{}{-n,\lambda+iy}{2\lambda}\bigg|1-e^{-2i\phi}\right),
\]
where $(x)_{n}:=x(x+1)(x+2)\cdots(x+n-1)$ is the Pochhammer symbol
and $_{2}F_{1}$ is the hypergeometric function \cite[Entries 15.1.1 and 15.2.1]{NIST}.
\end{defn}
Following a similar computation as that in, e.g., \cite[p.~1]{Zeng},
we see
\begin{equation}
Q_{n}(y):=i^{n}n!P_{n}^{\left(\frac{1}{2}\right)}\left(\frac{-iy}{2};\frac{\pi}{2}\right).\label{eq:OrthEuler}
\end{equation}
Two important properties of $P_{n}^{(\lambda)}(y;\phi)$ are listed
in the following proposition.
\begin{prop}
The Meixner-Pollaczek polynomials $P_{n}^{(\lambda)}(y;\phi)$ %
{} satify the recurrence 
\begin{equation}
(n+1)P_{n+1}^{(\lambda)}(y;\phi)=2(y\sin\phi+(n+\lambda)\cos\phi)P_{n}^{(\lambda)}(y;\phi)-(n+2\lambda-1)P_{n-1}^{(\lambda)}(y;\phi);\label{eq:MPRec}
\end{equation}
and the convolution formula 
\begin{equation}
P_{n}^{\left(\lambda+\mu\right)}\left(y_{1}+y_{2},\phi\right)=\sum_{k=0}^{n}P_{k}^{\left(\lambda\right)}\left(y_{1},\phi\right)P_{n-k}^{(\mu)}\left(y_{2},\phi\right).\label{eq:MPConvolution}
\end{equation}
\end{prop}
For proofs and further details of the proposition above, see \cite[eq.~9.7.3]{OrthPoly}
and \cite[p.~17]{Convolution}, respectively.

\subsection{Proof of Theorem \ref{thm:Orthogonal}. }

The following theorem gives an explicit expression of $\Omega_{n}^{(p)}(y)$,
which implies Theorem \ref{thm:Orthogonal}, by (\ref{eq:MPRec}).
\begin{thm}
\label{thm:MainResult}For positive integers $n$ and $p$, we have
\[
\Omega_{n}^{(p)}(y)=\frac{i^{n}n!}{2^{n}}P_{n}^{\left(\frac{p}{2}\right)}\left(-i\left(y-x+\frac{p}{2}\right);\frac{\pi}{2}\right).
\]
\end{thm}
To prove this result, we need the following lemma.
\begin{lem}
Let $Q_{n}^{(p)}(y)$ be the monic orthogonal polynomials with respect
to $\bar{E}_{n}^{(p)}$. Then, 
\begin{equation}
Q_{n}^{(p)}(y)=i^{n}n!P_{n}^{\left(\frac{p}{2}\right)}\left(-\frac{iy}{2};\frac{\pi}{2}\right).\label{eq:OrthEulerp}
\end{equation}
\end{lem}
\begin{proof}
We shall prove this by induction on the order $p$. Obviously, (\ref{eq:OrthEulerp})
coincides with (\ref{eq:OrthEuler}) when $p=1$. If $p>1$, we write
$p=p_{1}+p_{2}$, with both $p_{1},\ p_{2}\geq1$. By inductive hypothesis,
$Q_{n}^{(p_{i})}(y)$, defined by (\ref{eq:OrthEulerp}), are the
monic orthogonal polynomials with respect to $\bar{E}_{n}^{(p_{i})}$,
for $i=1,\ 2$. Now, we write $y=y_{1}+y_{2}$, so that by (\ref{eq:MPConvolution}),
\begin{align*}
 & \sum_{k=0}^{n}\binom{n}{k}Q_{k}^{(p_{1})}(y_{1})Q_{n-k}^{(p_{2})}(y_{2})\\
= & \sum_{k=0}^{n}\binom{n}{k}\left[i^{k}k!P_{k}^{\left(\frac{p_{1}}{2}\right)}\left(-\frac{iy_{1}}{2};\frac{\pi}{2}\right)\right]\left[i^{n-k}(n-k)!P_{n-k}^{\left(\frac{p_{2}}{2}\right)}\left(-\frac{iy_{2}}{2};\frac{\pi}{2}\right)\right]\\
= & i^{n}n!\sum_{k=0}^{n}P_{k}^{\left(\frac{p_{1}}{2}\right)}\left(-\frac{iy_{1}}{2};\frac{\pi}{2}\right)P_{n-k}^{\left(\frac{p_{2}}{2}\right)}\left(-\frac{iy_{2}}{2};\frac{\pi}{2}\right)\\
= & i^{n}n!P_{n}^{\left(\frac{p}{2}\right)}\left(-\frac{iy}{2};\frac{\pi}{2}\right)=Q_{n}^{(p)}(y).
\end{align*}
To show the orthogonality, we consider integers $r$ and $n$ with
$0\leq r<n$. Then
\begin{align*}
y^{r}Q_{n}^{(p)}(y) & =\left(y_{1}+y_{2}\right)^{r}\sum_{k=0}^{n}\binom{n}{k}Q_{k}^{(p_{1})}(y_{1})Q_{n-k}^{(p_{2})}(y_{2})\\
 & =\left(\sum_{l=0}^{n}\binom{n}{l}y_{1}^{l}y_{2}^{r-l}\right)\left(\sum_{k=0}^{n}\binom{n}{k}Q_{k}^{(p_{1})}(y_{1})Q_{n-k}^{(p_{2})}(y_{2})\right)\\
 & =\sum_{l=0}^{n}\sum_{k=0}^{n}\binom{n}{l}\binom{n}{k}\left(y_{1}^{l}Q_{k}^{(p_{1})}(y_{1})\right)\left(y_{2}^{r-l}Q_{n-k}^{(p_{2})}(y_{2})\right).
\end{align*}
Now, from (\ref{eq:Convolution}) and the fact that $\bar{E}_{n}^{(p)}=\mathbb{E}\left[\left(2\epsilon^{(p)}\right)^{n}\right]=\mathbb{E}\left[\left(2\epsilon^{(p_{1})}+2\epsilon^{(p_{2})}\right)^{n}\right]$,
we have
\begin{align*}
 & y^{r}Q_{n}^{(p)}(y)|_{y^{s}=\bar{E}_{s}^{(p)}}\\
= & \left(y_{1}+y_{2}\right)^{r}Q_{n}^{(p)}(y_{1}+y_{2})|_{y_{1}^{s}=\bar{E}_{s}^{(p_{1})},y_{2}^{s}=\bar{E}_{s}^{(p_{2})}}\\
= & \sum_{l=0}^{n}\sum_{k=0}^{n}\binom{n}{l}\binom{n}{k}\left(y_{1}^{l}Q_{k}^{(p_{1})}(y_{1})\right)\bigg|_{y_{1}^{s}=\bar{E}_{s}^{(p_{1})}}\left(y_{2}^{r-l}Q_{n-k}^{(p_{2})}(y_{2})\right)\bigg|_{y_{2}^{s}=\bar{E}_{s}^{(p_{2})}}.
\end{align*}
Since $l+(r-l)=r<n=k+(n-k)$ for each term in the sum above, either
$l<k$ or $r-l<n-k$ holds, implying the orthogonality: $y^{r}Q_{n}^{(p)}(y)|_{y^{s}=\bar{E}_{s}^{(p)}}=0$.
\end{proof}
\begin{proof}
[Proof of Theorem \ref{thm:MainResult}] From (\ref{eq:MomentEpn})
and (\ref{eq:EulerPN}), we see
\[
E_{n}^{(p)}(x)=\mathbb{E}\left[\left(x+\epsilon^{(p)}-\frac{p}{2}\right)^{n}\right]=\mathbb{E}\left[\left(x-\frac{p}{2}+\frac{1}{2}\cdot\left(2\epsilon^{(p)}\right)\right)^{n}\right],
\]
where, as shown above, $\bar{E}_{n}^{(p)}=\mathbb{E}\left[\left(2\epsilon^{(p)}\right)^{n}\right]$.
Then, we apply Lemma \ref{lem:KEY} twice, for $C=1/2$ and $c=x-p/2$,
to obtain
\[
\Omega_{n}^{(p)}(y)=\frac{1}{2^{n}}Q_{n}^{(p)}\left(2\left(y-x+\frac{p}{2}\right)\right)=\frac{i^{n}n!}{2^{n}}P_{n}^{\left(\frac{p}{2}\right)}\left(-i\left(y-x+\frac{p}{2}\right);\frac{\pi}{2}\right),
\]
which completes the proof.
\end{proof}
To conclude this section, we present an example to illustrate the
orthogonality relation (\ref{eq:OrthogonalityOmege}). 
\begin{example}
\label{exa:Orthogonality}When $p=2$, we see by (\ref{eq:RecOmega})
\[
\Omega_{2}^{(2)}(y)=\left(y-x+\frac{2}{2}\right)^{2}+\frac{1(1+2-1)}{4}=y^{2}-2\left(x-1\right)y+\left(x-1\right)^{2}+\frac{1}{2}.
\]
Using Table \ref{tab:EpnFirstValues}, we have

\begin{align*}
y^{0}\Omega_{2}^{(2)}(y)|_{y^{k}=E_{k}^{(2)}(x)} & =y^{2}-2\left(x-1\right)y+\left(x-1\right)^{2}+\frac{1}{2}|_{y^{k}=E_{k}^{(2)}(x)}\\
 & =x^{2}-2x+\frac{1}{2}-2\left(x-1\right)^{2}+\left(x-1\right)^{2}+\frac{1}{2}=0,
\end{align*}
and similarly 
\begin{align*}
 & y\Omega_{2}^{(2)}(y)|_{y^{k}=E_{k}^{(2)}(x)}\\
= & y^{3}-2\left(x-1\right)y^{2}+\left(x-1\right)^{2}y+\frac{y}{2}|_{y^{k}=E_{k}^{(2)}(x)}\\
= & x^{3}-3x^{2}+\frac{3}{2}x+\frac{1}{2}-2\left(x-1\right)\left(x^{2}-2x+\frac{1}{2}\right)+\left(x-1\right)^{3}+\frac{x-1}{2}=0.
\end{align*}
This confirms (\ref{eq:OrthogonalityOmege}) for $p=n=2$. 
\end{example}

\section{\label{sec:GMN}Connection to generalized Motzkin numbers}

\subsection{Preliminaries. }

Recall the definition of the generalized Motzkin numbers in Definition
\ref{def:GMN}. In fact, $M_{n,k}$ counts the number of certain weighted
lattice paths, called \emph{Motzkin paths} \cite[p.~319]{AC}. More
specifically, consider the paths with the following restrictions: 

1. all paths lie within the first quadrant;

2. only allow three types of paths:

~~~~a) horizontal path $\alpha_{k}$ from $(j,k)$ to $(j+1,k);$

~~~~b) diagonally up path $\beta_{k}$ from $(j,k)$ to $(j+1,k+1);$

~~~~c) and diagonally down path $\gamma_{k}$ from $(j,k)$ to
$(j+1,k-1);$ 

3. associate each type of the paths to different weights as $\alpha_{k}\mapsto1$,
$\beta_{k}\mapsto\sigma_{k}$, and 

~~~~$\gamma_{k}\mapsto\tau_{k}$, which we shall denote by the
weight triple $\left(1,\sigma_{k},\tau_{k}\right)$.

Then, $M_{n,k}$ counts the number of $(1,\sigma_{k},\tau_{k})$-weighted
paths from $(0,0)$ to $(n,k)$. 
\begin{example}
When $\sigma_{k}=\tau_{k}=1$, $M_{n,k}$ is the (usual) Motzkin number
that counts the number of paths from $(0,0)$ to $(n,k)$, with only
the first two restrictions above. For instance, when $n=3$ and $k=1$,
by recurrence (\ref{eq:MotzkinNumberRec}), we have
\begin{align*}
M_{3,1} & =M_{2,0}+M_{2,1}+M_{2,2}\\
 & =\left(M_{1,-1}+M_{1,0}+M_{1,1}\right)+\left(M_{1,0}+M_{1,1}+M_{1,2}\right)+\left(M_{1,1}+M_{1,2}+M_{1,3}\right)\\
 & =2M_{1,0}+3M_{1,1}\\
 & =2\left(M_{0,-1}+M_{0,0}+M_{0,1}\right)+3\left(M_{0,0}+M_{0,1}+M_{0,2}\right)=5.
\end{align*}
All the five paths from $(0,0)$ to $(3,1)$ are listed as follows.

\includegraphics{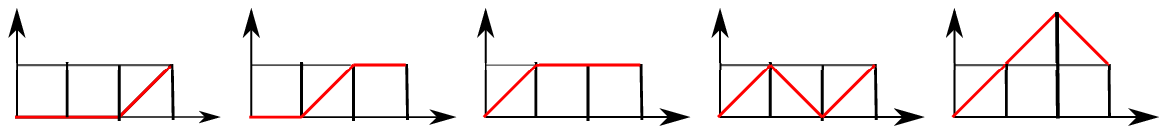}
\end{example}
The next result shows that when $k=0$, the generating function of
$M_{n,0}$ admits a form of continued fractions (called \emph{J(acobi)-fractions}).
See, e.g., \cite[p.~324]{AC}.
\begin{thm}
\label{thm:CFGMN}For the generalized Motzkin numbers $M_{n,k}$ defined
by (\ref{eq:MotzkinNumberRec}), we have
\begin{equation}
\sum_{n=0}^{\infty}M_{n,0}z^{n}=\frac{1}{1-\sigma_{0}z-\frac{\tau_{1}z^{2}}{1-\sigma_{1}z-\frac{\tau_{2}z^{2}}{1-\sigma_{2}z-\cdots}}}.\label{eq:JFractionMotzkin}
\end{equation}
\end{thm}
A similar expression is known for moments. See, e.g., \cite[pp.~20--21]{Determinant}.
\begin{thm}
\label{thm:CFMoments} Let $X$ be an arbitrary random variable, with
moments $m_{n}$ and monic orthogonal polynomials $P_{n}(y)$ satisfying
the recurrence (\ref{eq:ThreeTermRec}), involving two sequences $s_{n}$
and $t_{n}$. Then, we have
\begin{equation}
\sum_{n=0}^{\infty}m_{n}z^{n}=\frac{m_{0}}{1-s_{0}z-\frac{t_{1}z^{2}}{1-s_{1}z-\frac{t_{2}z^{2}}{1-s_{2}z-\cdots}}}.\label{eq:JFractionMoment}
\end{equation}
\end{thm}
Combining (\ref{eq:JFractionMoment}) and (\ref{eq:JFractionMotzkin})
leads to the following general theorem.
\begin{thm}
\label{thm:generalTHM}Let the two sequences $s_{n}$ and $t_{n}$
be the ones appearing in the recurrence (\ref{eq:ThreeTermRec}),
for some random variable X, and assume that $m_{0}=1$. Also define
the generalized Motzkin number sequence $M_{n,k}$ by letting $\sigma_{k}=s_{k}$
and $\tau_{k}=t_{k}$ in (\ref{eq:MotzkinNumberRec}). Then, $M_{n,0}$
gives the moments of $X$, i.e., $M_{n,0}=m_{n}=\mathbb{E}[X^{n}]$. 
\end{thm}
\begin{rem*}
The condition $m_{0}=1$ in Theorem \ref{thm:generalTHM} is usually
guaranteed by normalization of the density function.
\end{rem*}

\subsection{Proof and Applications of Theorem \ref{thm:Motzkin}. }
\begin{proof}
[Proof of Theorem \ref{thm:Motzkin}.]We apply Theorem \ref{thm:generalTHM}
to the random variable $\mathcal{E}^{(p)}+x$, whose moments are $E_{n}^{(p)}(x)$,
we directly prove Theorem \ref{thm:Motzkin}. 
\end{proof}
The combinatorial interpretation for $E_{n}^{(p)}(x)$ can now be
used to to obtain the following matrix representation.
\begin{thm}
\label{thm:MatrixEuler}Define the infinite dimensional matrix 
\[
RE^{(p)}:=\begin{pmatrix}x-\frac{p}{2} & -\frac{p}{4} & 0 & 0 & \cdots & 0 & \cdots\\
1 & x-\frac{p}{2} & -\frac{p+1}{2} & 0 & \cdots & 0 & \cdots\\
0 & 1 & x-\frac{p}{2} & \ddots & \ddots & \vdots & \cdots\\
0 & 0 & 1 & \ddots & -\frac{n(n+p-1)}{4} & 0 & \cdots\\
\vdots & \vdots & \vdots & \ddots & x-\frac{p}{2} & -\frac{(n+1)(n+p)}{4} & \cdots\\
0 & 0 & 0 & \ddots & 1 & \ddots & \ddots\\
\vdots & \vdots & \vdots & \cdots & \vdots & \ddots & \ddots
\end{pmatrix},
\]
and let $RE_{m}^{(p)}$ be the left upper $m\times m$ block of $RE^{(p)}$.
Then for any nonnegative integer $n\leq m$, the left upper entry
of $\left(RE_{m}^{(p)}\right)^{n}$ gives $E_{n}^{(p)}(x)$, i.e.,
\[
\left[\left(RE_{m}^{(p)}\right)^{n}\right]_{1,1}=E_{n}^{(p)}(x).
\]
\end{thm}
\begin{example}
Let $p=2$ and $m=4$. Then
\[
RE_{4}^{(2)}:=\begin{pmatrix}x-1 & -1/2 & 0 & 0\\
1 & x-1 & -3/2 & 0\\
0 & 1 & x-1 & -3\\
0 & 0 & 1 & x-1
\end{pmatrix}
\]
and 
\[
\left(RE_{4}^{(2)}\right)^{3}=\begin{pmatrix}x^{3}-3x^{2}+\frac{3}{2}x+\frac{1}{2} & * & * & *\\
* & * & * & *\\
* & * & * & *\\
* & * & * & *
\end{pmatrix},
\]
confirming $E_{3}^{(2)}(x)=x^{3}-3x^{2}+3x/2+1/2$. (See Table \ref{tab:EpnFirstValues}.)
\end{example}
\begin{example}
Recall $E_{n}^{(1)}(1/2)=E_{n}(1/2)=E_{n}/2^{n}$, which gives the
Motzkin paths weighted by $(1,0,-k^{2}/4)$, which means that the
horizontal paths are eliminated. Therefore, $E_{n}/2^{n}$ counts
the weighted \emph{Dyck paths}, related to Catalan numbers $C_{n}$
\cite[Ex.~25]{Catalan}. For example, when $n=6$, there are 
\[
C_{3}:=\frac{1}{4}\binom{6}{3}=5
\]
weighted Dyck paths, listed as follows:

\includegraphics[scale=0.55]{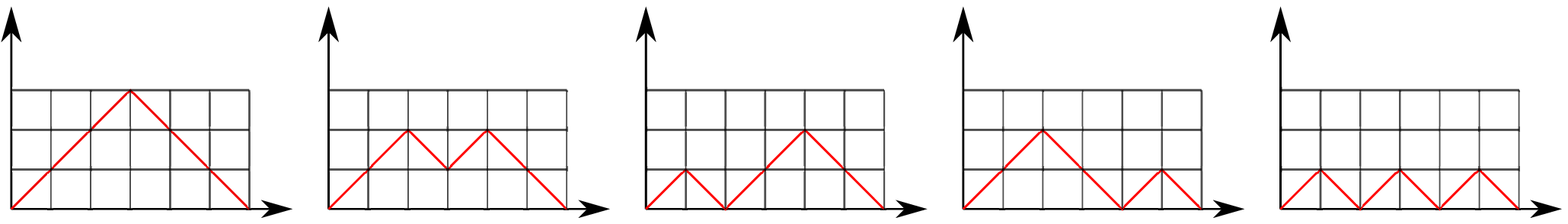}

\noindent Then, by noting that each diagonally down path from $(j,k)$
to $(j+1,k-1)$ has weight $-k^{2}/4$, we have 
\[
-\frac{61}{64}=\frac{E_{6}}{2^{6}}=\left(\frac{-1}{4}\right)^{3}\left(3^{2}2^{2}1^{2}+2^{2}2^{2}1^{2}+1^{2}2^{2}1^{2}+2^{2}1^{2}1^{2}+1^{2}1^{2}1^{2}\right).
\]
\end{example}
\begin{rem*}
From the computation above, one might have noticed and could prove,
by Lemma \ref{lem:KEY}, that the Euler numbers $E_{n}$ are given
by the weighted lattice paths $(1,0,-k^{2})$. This reconfirms that
Euler numbers are integers, odd index terms vanish, and even terms
have alternating signs. (See also \cite[Entries 24.2.7 and 24.2.9]{NIST}.)
\end{rem*}

\section{\label{sec:Bernoulli}Analogue to Bernoulli polynomials}

The \emph{higher-order Bernoulli polynomials} $B_{n}^{(p)}(x)$ are
defined by 
\[
\left(\frac{z}{e^{z}-1}\right)^{p}e^{zx}=\sum_{n=0}^{\infty}B_{n}^{(p)}(x)\frac{z^{n}}{n!}.
\]
When $p=1$, $B_{n}^{(1)}(x)=B_{n}(x)$ is the \emph{Bernoulli polynomial}
and $B_{n}(0)=B_{n}$ is the \emph{Bernoulli number}. Probabilistic
interpretation for $B_{n}(x)$ can be found, e.g., \cite[eq.~2.14]{Zagier}.
Touchard \cite[eq.~44]{Touchard} computed the orthogonal polynomials
with respect to Bernoulli numbers, denoted by $R_{n}(y)$. More specifically,
$R_{0}(y)=1$, $R_{1}(y)=y+1/2$ and for $n\geq1$, 
\[
R_{n+1}(y)=\left(y+\frac{1}{2}\right)R_{n}(y)-\frac{n^{4}}{4(2n+1)(2n-1)}R_{n-1}(y).
\]
Following similar steps, we shall obtain analogues of Theorem \ref{thm:Orthogonal}
and Theorem \ref{thm:Motzkin} for Bernoulli polynomials. The proof
is omitted. 
\begin{thm}
Let $\varrho_{n}(y)$ be the orthogonal polynomials with respect to
$B_{n}(x)$, i.e., for integers $r$ and $n$, with $0\leq r<n$,
\[
y^{r}\varrho_{n}(y)|_{y^{k}=B_{k}(x)}=0.
\]
Then, $\varrho_{0}(y)=1$, $\varrho_{1}(y)=y-x+1/2$ and for $n\geq1$,
\begin{equation}
\varrho_{n+1}(y)=\left(y-x+\frac{1}{2}\right)\varrho_{n}(y)-\frac{n^{4}}{4(2n+1)(2n-1)}\varrho_{n-1}(y).\label{eq:RecB}
\end{equation}
In particular, 
\[
\varrho_{n}(y)=\frac{n!}{(n+1)_{n}}p_{n}\left(y;\frac{1}{2},\frac{1}{2},\frac{1}{2},\frac{1}{2}\right),
\]
where $p_{n}(y;a,b,c,d)$ is the continuous-Hahn polynomial \cite[pp.~200--202]{OrthPoly}.
Moreover, let $\mathfrak{B}_{n,k}$ be the generalized Motzkin numbers
with special choice $\sigma_{k}=x-1/2$ and $\tau_{k}=-k^{4}/(4(2k+1)(2k-1))$.
Then $\mathfrak{B}_{n,0}=B_{n}(x)$. The matrix 

\[
RB:=\begin{pmatrix}x-\frac{1}{2} & -\frac{1}{12} & 0 & 0 & \cdots & 0 & \cdots\\
1 & x-\frac{1}{2} & -\frac{4}{15} & 0 & \cdots & 0 & \cdots\\
0 & 1 & x-\frac{1}{2} & \ddots & \ddots & \vdots & \cdots\\
0 & 0 & 1 & \ddots & -\frac{n^{4}}{4(2n+1)(2n-1)} & 0 & \cdots\\
\vdots & \vdots & \vdots & \ddots & x-\frac{1}{2} & -\frac{(n+1)^{2}}{4(2n+1)(2n+3)} & \cdots\\
0 & 0 & 0 & \ddots & 1 & \ddots & \ddots\\
\vdots & \vdots & \vdots & \cdots & \vdots & \ddots & \ddots
\end{pmatrix}
\]
generate all $B_{n}(x)$ through the power of its left upper block,
as an analogue to Theorem \ref{thm:MatrixEuler}.
\end{thm}
\begin{rem*}
We also tried to consider the analogue on $B_{n}^{(p)}(x)$. However,
the rational coefficients $n^{4}/(4(2n+1)(2n-1))$ provides more difficulties
than $n^{2}$ that appears in (\ref{eq:RecQ}). Moreover, the convolution
property (\ref{eq:MPConvolution}) for Meixner-Pollaczek polynomials
fails for continuous Hahn polynomials. Therefore, we only have the
following computation and conjectures.
\end{rem*}
Let $\varrho_{n+1}^{(p)}(y)$ be the monic orthogonal polynomial with
respect to $B_{n}^{(p)}(x)$, and assume the three-term recurrence
is
\[
\varrho_{n+1}^{(p)}(y)=\left(y-a_{n}^{(p)}\right)\varrho_{n}^{(p)}(y)-b_{n}^{(p)}\varrho_{n-1}(y).
\]
To compute $\varrho_{n}^{(p)}(y)$, one could use, e.g., \cite[eq.~2.1.10]{Ismail},
which does not give explicit formulas for $a_{n}^{(p)}$ and $b_{n}^{(p)}$.
However, for $a_{n}^{(p)}$, Lemma \ref{lem:KEY} implies that 
\[
a_{n}^{(p)}=x-p/2.
\]
The first several terms of $b_{n}^{(p)}$ is given in the following
table

\vspace{-8bp}

\begin{table}[H]
\begin{centering}
{\setlength\extrarowheight{2pt}%
\begin{tabular}{|l|c|c|c|c|c|}
\hline 
\multirow{1}{*}{} & $p=1$  & $p=2$  & $p=3$  & $p=4$ & $p=5$\tabularnewline
\hline 
$n=1$  & $\frac{1}{12}$ & $\frac{1}{6}$  & $\frac{1}{4}$  & $\frac{1}{3}$ & $\frac{5}{12}$\tabularnewline
$n=2$  & $\frac{4}{15}$ & $\frac{13}{30}$  & $\frac{3}{5}$ & $\frac{23}{30}$ & $\frac{14}{15}$\tabularnewline
$n=3$  & $\frac{81}{140}$ & $\frac{372}{455}$ & $\frac{1339}{1260}$ & $\frac{2109}{1610}$ & $\frac{1527}{980}$\tabularnewline
$n=4$  & $\frac{64}{63}$ & $\frac{3736}{2821}$  & $\frac{138688}{84357}$ & $\frac{668543}{339549}$ & $\frac{171830}{74823}$\tabularnewline
$n=5$ & $\frac{625}{396}$ & $\frac{1245075}{636988}$  & $\frac{299594775}{127670972}$ & $\frac{42601023200}{15509529057}$ & $\frac{3638564965}{1154491404}$\tabularnewline
\hline 
\end{tabular}}\vspace{-10bp}
\par\end{centering}
\caption{\label{tab:bpn}$b_{n}^{(p)}$ for $1\leq n,\ p\leq5$}
\end{table}

\vspace{-10bp}

Here, the first column is $b_{n}^{(1)}=n^{4}/(4(2n+1)(2n-1))$, as
that in the last term of (\ref{eq:RecB}). Also, one can easily see
that the first row is linear as $b_{1}^{(p)}=p/12$, so is the second
row $b_{2}^{(p)}=(5p+3)/10$. The following conjecture is due to Karl
Dilcher.
\begin{conjecture}
The third row is given by
\[
b_{3}^{(p)}=\frac{175p^{2}+315p+158}{140(2p+3)};
\]
the fourth row satisfies 
\[
b_{4}^{(p)}=\frac{6125p^{4}+25725p^{3}+41965p^{2}+29547p+7230}{21(5p+3)(175p^{2}+315p+158)};
\]
and the fifth row is 

$b_{5}^{(p)}=25(5p+3)(471625p^{6}+3678675p^{5}+12324235p^{4}+22096305p^{3}+22009540p^{2}+11549748p+2519472)\bigg/(132(175p^{2}+315p+158)(6125p^{4}+25725p^{3}+41965p^{2}+29547p+7230))$.
\end{conjecture}
\begin{rem*}
We do not have a conjecture on the general formula for $b_{n}^{(p)}$. 
\end{rem*}

\section*{Acknowledgment}

The corresponding author is supported by the National Science Foundation
of China (No.~1140149).

The first author is supported by the Izaak Walton Killam Postdoctoral
Fellowship at Dalhousie University. He would like to thank his supervisor
Prof.~Karl Dilcher and also Prof.~Christophe Vignat for their valuable
suggestions and guidance.

\end{document}